\documentclass{conm-p-l}
\usepackage{amssymb}

\usepackage{amsfonts}
\usepackage{eucal}
\usepackage{upgreek}
\begin{document}

\font\eightrm=cmr8
\font\eightit=cmti8
\font\eighttt=cmtt8
\font\tensans=cmss10
\def\emp{\text{\tensans\O}}
\def\tci
{\hbox{\hskip1.8pt$\rightarrow$\hskip-11.5pt$^{^{C^\infty}}$\hskip-1.3pt}}
\def\nft
{\hbox{$n$\hskip3pt$\equiv$\hskip4pt$5$\hskip4.4pt$($mod\hskip2pt$3)$}}
\def\bbR{\mathrm{I\!R}}
\def\rto{\bbR\hskip-.5pt^2}
\def\rtr{\bbR\hskip-.7pt^3}
\def\rfo{\bbR\hskip-.7pt^4}
\def\rn{\bbR^{\hskip-.6ptn}}
\def\mr{\bbR^{\hskip-.6ptm}}
\def\bbZ{\mathsf{Z\hskip-4ptZ}}
\def\bbRP{\text{\bf R}\text{\rm P}}
\def\bbC{{\mathchoice {\setbox0=\hbox{$\displaystyle\rm C$}\hbox{\hbox
to0pt{\kern0.4\wd0\vrule height0.9\ht0\hss}\box0}}
{\setbox0=\hbox{$\textstyle\rm C$}\hbox{\hbox
to0pt{\kern0.4\wd0\vrule height0.9\ht0\hss}\box0}}
{\setbox0=\hbox{$\scriptstyle\rm C$}\hbox{\hbox
to0pt{\kern0.4\wd0\vrule height0.9\ht0\hss}\box0}}
{\setbox0=\hbox{$\scriptscriptstyle\rm C$}\hbox{\hbox
to0pt{\kern0.4\wd0\vrule height0.9\ht0\hss}\box0}}}}
\def\bbQ{{\mathchoice {\setbox0=\hbox{$\displaystyle\rm Q$}\hbox{\raise
0.15\ht0\hbox to0pt{\kern0.4\wd0\vrule height0.8\ht0\hss}\box0}}
{\setbox0=\hbox{$\textstyle\rm Q$}\hbox{\raise
0.15\ht0\hbox to0pt{\kern0.4\wd0\vrule height0.8\ht0\hss}\box0}}
{\setbox0=\hbox{$\scriptstyle\rm Q$}\hbox{\raise
0.15\ht0\hbox to0pt{\kern0.4\wd0\vrule height0.7\ht0\hss}\box0}}
{\setbox0=\hbox{$\scriptscriptstyle\rm Q$}\hbox{\raise
0.15\ht0\hbox to0pt{\kern0.4\wd0\vrule height0.7\ht0\hss}\box0}}}}
\def\bbQ{{\mathchoice {\setbox0=\hbox{$\displaystyle\rm Q$}\hbox{\raise
0.15\ht0\hbox to0pt{\kern0.4\wd0\vrule height0.8\ht0\hss}\box0}}
{\setbox0=\hbox{$\textstyle\rm Q$}\hbox{\raise
0.15\ht0\hbox to0pt{\kern0.4\wd0\vrule height0.8\ht0\hss}\box0}}
{\setbox0=\hbox{$\scriptstyle\rm Q$}\hbox{\raise
0.15\ht0\hbox to0pt{\kern0.4\wd0\vrule height0.7\ht0\hss}\box0}}
{\setbox0=\hbox{$\scriptscriptstyle\rm Q$}\hbox{\raise
0.15\ht0\hbox to0pt{\kern0.4\wd0\vrule height0.7\ht0\hss}\box0}}}}
\def\dimr{\dim_{\hskip.4pt\bbR\hskip-1.2pt}\w}
\def\dimc{\dim_{\hskip.4pt\bbC\hskip-1.2pt}\w}
\def\aff{\mathrm{A\hn f\hh f}\hs}
\def\cx{C\hskip-2pt_x\w}
\def\cy{C\hskip-2pt_y\w}
\def\cz{C\hskip-2pt_z\w}
\def\hyp{\hskip.5pt\vbox
{\hbox{\vrule width3ptheight0.5ptdepth0pt}\vskip2.2pt}\hskip.5pt}
\def\er{r}
\def\es{s}
\def\df{d\hskip-.8ptf}
\def\dz{\mathcal{D}}
\def\dzp{\dz^\perp}
\def\fv{\mathcal{F}}
\def\gr{\mathcal{G}}
\def\fvp{\fv_{\nrmh p}}
\def\wv{\mathcal{W}}
\def\vt{\mathcal{P}}
\def\tv{\mathcal{T}}
\def\vr{\mathcal{V}}
\def\xs{J}
\def\xl{S}
\def\cs{\mathcal{B}}
\def\zy{\mathcal{Z}}
\def\vtx{\vt_{\nh x}}
\def\fh{f}
\def\g{\mathtt{g}}
\def\rc{\theta}
\def\jm{\mathcal{I}}
\def\ke{\mathcal{K}}
\def\xc{\mathcal{X}_c}
\def\lz{\mathcal{L}}
\def\dla{\mathcal{D}_{\hskip-2ptL}^*}
\def\Lie{\pounds}
\def\lv{\Lie\hskip-1.2pt_v\w}
\def\lo{\lz_0}
\def\xe{\mathcal{E}}
\def\eo{\xe_0}
\def\lsq{\mathsf{[}}
\def\rsq{\mathsf{]}}
\def\hga{\hskip2.3pt\widehat{\hskip-2.3pt\gamma\hskip-2pt}\hskip2pt}
\def\hm{\hskip1.9pt\widehat{\hskip-1.9ptM\hskip-.2pt}\hskip.2pt}
\def\hg{\hskip.9pt\widehat{\hskip-.9pt\g\hskip-.9pt}\hskip.9pt}
\def\hna{\hskip.2pt\widehat{\hskip-.2pt\nabla\hskip-1.6pt}\hskip1.6pt}
\def\hdz{\hskip.9pt\widehat{\hskip-.9pt\dz\hskip-.9pt}\hskip.9pt}
\def\hdp{\hskip.9pt\widehat{\hskip-.9pt\dz\hskip-.9pt}\hskip.9pt^\perp}
\def\hmt{\hskip1.9pt\widehat{\hskip-1.9ptM\hskip-.5pt}_t}
\def\hmz{\hskip1.9pt\widehat{\hskip-1.9ptM\hskip-.5pt}_0}
\def\hmp{\hskip1.9pt\widehat{\hskip-1.9ptM\hskip-.5pt}_p}
\def\hk{\hskip1.5pt\widehat{\hskip-1.5ptK\hskip-.5pt}\hskip.5pt}
\def\hq{\hskip1.5pt\widehat{\hskip-1.5ptQ\hskip-.5pt}\hskip.5pt}
\def\txm{{T\hskip-2.9pt_x\w\hn M}}
\def\tyhm{{T\hskip-3.5pt_y\w\hm}}
\def\q{q}
\def\bq{\hat q}
\def\p{p}
\def\w{^{\phantom i}}
\def\x{u}
\def\y{v}
\def\vp{{\tau\hskip-4.55pt\iota\hskip.6pt}}
\def\evp{{\tau\hskip-3.55pt\iota\hskip.6pt}}
\def\vd{\vt\hh'}
\def\vdx{\vd{}\hskip-4.5pt_x}
\def\bz{b\hh}
\def\fe{F}
\def\fy{\phi}
\def\vl{\Lambda}
\def\hy{\mathcal{V}}
\def\vh{h}
\def\bc{C}
\def\mv{V}
\def\vo{V_{\nnh0}}
\def\ao{A_0}
\def\bo{B_0}
\def\uv{\mathcal{U}}
\def\sv{\mathcal{S}}
\def\svp{\sv_p}
\def\xv{\mathcal{X}}
\def\xvp{\xv_p}
\def\yv{\mathcal{Y}}
\def\yvp{\yv_p}
\def\zv{\mathcal{Z}}
\def\zvp{\zv_p}
\def\cv{\mathcal{C}}
\def\dy{\mathcal{D}}
\def\nv{\mathcal{N}}
\def\iv{\mathcal{I}}
\def\sgp{\sigma\hskip-1.8pt_p\w}
\def\gkp{\Sigma}
\def\ret{\sigma}
\def\taw{\uptau}
\def\hs{\hskip.7pt}
\def\hh{\hskip.4pt}
\def\hn{\hskip-.4pt}
\def\nh{\hskip-.7pt}
\def\nnh{\hskip-1pt}
\def\hrz{^{\hskip.5pt\text{\rm hrz}}}
\def\vrt{^{\hskip.2pt\text{\rm vrt}}}
\def\vt{\varTheta}
\def\mtr{\Theta}
\def\op{\varTheta}
\def\vg{\varGamma}
\def\my{\mu}
\def\ny{\nu}
\def\gy{\lambda}
\def\lp{\lambda}
\def\ax{\alpha}
\def\lf{\widetilde{\lp}}
\def\bx{\beta}
\def\ay{a}
\def\by{b}
\def\gp{\mathrm{G}}
\def\hp{\mathrm{H}}
\def\kp{\mathrm{K}}
\def\gm{\gamma}
\def\Gm{\Gamma}
\def\Lm{\Lambda}
\def\Dt{\Delta}
\def\dg{\Delta}
\def\sj{\sigma}
\def\lg{\langle}
\def\rg{\rangle}
\def\lr{\langle\hh\cdot\hs,\hn\cdot\hh\rangle}
\def\vs{vector space}
\def\rvs{real vector space}
\def\vf{vector field}
\def\tf{tensor field}
\def\tvn{the vertical distribution}
\def\dn{distribution}
\def\pt{point}
\def\tc{tor\-sion\-free connection}
\def\ea{equi\-af\-fine}
\def\rt{Ric\-ci tensor}
\def\pde{partial differential equation}
\def\pf{projectively flat}
\def\pfs{projectively flat surface}
\def\pfc{projectively flat connection}
\def\pftc{projectively flat tor\-sion\-free connection}
\def\su{surface}
\def\sco{simply connected}
\def\psr{pseu\-\hbox{do\hs-}Riem\-ann\-i\-an}
\def\inv{-in\-var\-i\-ant}
\def\trinv{trans\-la\-tion\inv}
\def\feo{dif\-feo\-mor\-phism}
\def\feic{dif\-feo\-mor\-phic}
\def\feicly{dif\-feo\-mor\-phi\-cal\-ly}
\def\Feicly{Dif-feo\-mor\-phi\-cal\-ly}
\def\diml{-di\-men\-sion\-al}
\def\prl{-par\-al\-lel}
\def\skc{skew-sym\-met\-ric}
\def\sky{skew-sym\-me\-try}
\def\Sky{Skew-sym\-me\-try}
\def\dbly{-dif\-fer\-en\-ti\-a\-bly}
\def\cf{con\-for\-mal\-ly flat}
\def\ls{locally symmetric}
\def\ecs{essentially con\-for\-mal\-ly symmetric}
\def\rr{Ric\-ci-re\-cur\-rent}
\def\kf{Kil\=ling field}
\def\om{\omega}
\def\vol{\varOmega}
\def\og{\varOmega\hs}
\def\dv{\delta}
\def\ve{\varepsilon}
\def\zt{\zeta}
\def\kx{\kappa}
\def\mf{manifold}
\def\mfd{-man\-i\-fold}
\def\bmf{base manifold}
\def\bd{bundle}
\def\tbd{tangent bundle}
\def\ctb{cotangent bundle}
\def\bp{bundle projection}
\def\prc{pseu\-\hbox{do\hs-}Riem\-ann\-i\-an metric}
\def\prd{pseu\-\hbox{do\hs-}Riem\-ann\-i\-an manifold}
\def\Prd{pseu\-\hbox{do\hs-}Riem\-ann\-i\-an manifold}
\def\npd{null parallel distribution}
\def\pj{-pro\-ject\-a\-ble}
\def\pd{-pro\-ject\-ed}
\def\lcc{Le\-vi-Ci\-vi\-ta connection}
\def\vb{vector bundle}
\def\vbm{vec\-tor-bun\-dle morphism}
\def\kerd{\text{\rm Ker}\hskip2.7ptd}
\def\ro{\rho}
\def\sy{\sigma}
\def\ts{total space}
\def\pmb{\pi}
\def\pl{\partial}
\def\cro{\overline{\hskip-2pt\pl}}
\def\ddb{\pl\hskip1.7pt\cro}
\def\dbd{\cro\hskip-.3pt\pl}
\newtheorem{theorem}{Theorem}[section] 
\newtheorem{proposition}[theorem]{Proposition} 
\newtheorem{lemma}[theorem]{Lemma} 
\newtheorem{corollary}[theorem]{Corollary} 
  
\theoremstyle{definition} 
  
\newtheorem{defn}[theorem]{Definition} 
\newtheorem{notation}[theorem]{Notation} 
\newtheorem{example}[theorem]{Example} 
\newtheorem{conj}[theorem]{Conjecture} 
\newtheorem{prob}[theorem]{Problem} 
  
\theoremstyle{remark} 
  
\newtheorem{remark}[theorem]{Remark}

\title[Af\-fine vector fields]{Af\-fine vector fields on compact
 pseu\-do\hs-\nh K\"ahler manifolds}
\author[A. Derdzinski]{Andrzej Derdzinski} 
\address{Department of Mathematics, The Ohio State University, 
Columbus, OH 43210, USA} 
\email{andrzej@math.ohio-state.edu} 
\dedicatory{Dedicated to Paolo Piccione on the occasion of his 60th birthday}
\thanks{The author wishes to thank Ivo Terek for very helpful discussions,
and the anonymous referee for numerous comments and suggestions that greatly
improved the exposition.}
\thanks{Research supported in part by a FAPESP\hn-\hh OSU 2015 Regular
Research Award (FAPESP grant: 2015/50265-6)}
\def\leftmark{A.\ Derdzinski}
\def\rightmark{Af\-fine vector fields}

\begin{abstract}
It is known that a Kil\-ling field on a compact
pseu\-do\hs-\hn K\"ahler manifold is necessarily (real) 
hol\-o\-mor\-phic, as long as the manifold satisfies some relatively mild
additional conditions. We provide two further proofs of this fact and discuss 
the natural open question whether the same conclusion holds for
af\-fine -- rather than Kil\-ling -- vector fields. The question cannot be
settled by invoking the Kil\-ling case: Boubel and Mounoud
[Trans.\,Amer.\,Math.\,Soc.\,368, 2016, 2223--2262] 
constructed examples of non-Kil\-ling af\-fine vector fields on
compact pseu\-do\hs-Riem\-ann\-i\-an manifolds. We show that an af\-fine
vector field $\,v\,$ is necessarily symplectic, and establish some algebraic 
and differential properties of the Lie derivative of the metric along $\,v$, 
such as its being parallel, anti\-lin\-e\-ar and
nil\-po\-tent as an en\-do\-mor\-phism of the tangent bundle. As a
consequence, the answer to the above question turns out to be 
`yes' whenever the underlying manifold admits no nontrivial hol\-o\-mor\-phic
quadratic differentials, which includes the case of
compact almost homogeneous complex manifolds
with nonzero Euler characteristic.
\end{abstract}

\maketitle

\setcounter{section}{0}
\setcounter{theorem}{0}
\renewcommand{\theequation}{\arabic{section}.\arabic{equation}}
\renewcommand{\thetheorem}{\Alph{theorem}}
\section*{Introduction}
\setcounter{equation}{0}
A {\it pseu\-do\hs-\hn K\"ahler manifold\/} is a pseu\-do\hs-Riem\-ann\-i\-an
manifold endowed with a parallel al\-most-com\-plex structure $\,J$,
making the metric Her\-mit\-i\-an. This is well known 
to imply in\-te\-gra\-bi\-li\-ty of $\,J$, via the New\-land\-er-Ni\-ren\-berg
theorem: the Nijen\-huis tensor $\,N\hs$ sends 
vector fields $\,v,w\,$ to the vector field 
$\,N(v,w)=[v,w]+J[Jv,w]+J[v,Jw]-[Jv,Jw]$, so that
$\,N(v,w)=[J\nabla_{\nnh\!v}\w\nnh J
-\nabla_{\!\!J\hn v}\w J]\hh w+[\nabla_{\!\!J\hn w}\w J
-J\nabla_{\nh\!w}\w J]\hh v$ for any tor\-sion-free 
connection $\,\nabla\nnh$. Thus, $\,N\nnh=0\,$ whenever $\,\nabla\nh J=0$.

It is also well known \cite[pp.\,60--61]{ballmann} that Kil\-ling fields
on compact {\it Riemannian\/} K\"ahler manifolds are necessarily (real)
hol\-o\-mor\-phic, compactness being essential (as illustrated by flat 
manifolds). This remains valid, under some additional assumptions, in the
pseu\-do\hs-\hn K\"ahler case \cite{schouten-yano,derdzinski-terek}.

A vector field $\,v\,$ on a manifold endowed with a connection $\,\nabla\hs$
is said to be {\it af\-fine\/} if its local flow preserves $\,\nabla\nh$.  
When $\,\nabla\hs$ is the Le\-vi-Ci\-vi\-ta connection of a Riemannian metric,
such $\,v\,$ are usually Kil\-ling fields, with very few -- always
noncompact -- exceptions \cite[Ch.\,IV]{kobayashi}. 
However, Boubel and Mounoud \cite{boubel-mounoud}
provided examples of compact pseu\-do\hs-Riem\-ann\-i\-an manifolds 
admitting non-Kil\-ling af\-fine vector fields. Their examples are not of
the pseu\-do\hs-\hn K\"ahler type, which raises a question: are the 
af\-fine-to-Kil\-ling and af\-fine-to-hol\-o\-mor\-phic implications true for
compact pseu\-do\hs-\hn K\"ahler manifolds?

This question remains open. The present paper establishes some 
properties of the Lie derivative
$\,\pounds\hskip-1.5pt_v\w g\,$ for an af\-fine vector field $\,v\,$ 
on a compact pseu\-do\hs-\hn K\"ahler manifold $\,(M\nh,g)\,$ with the
$\,\ddb$ property (Theorem~\ref{fourc}): in addition to being obviously
parallel, $\,\pounds\hskip-1.5pt_v\w g\,$ is also anti\-lin\-e\-ar and 
nil\-po\-tent as an en\-do\-mor\-phism of $\,T\nh M\nh$, while
$\,v\,$ is a Kil\-ling field if and only if it is real hol\-o\-mor\-phic.
Also, $\,\pounds\hskip-1.5pt_v\w\hs\omega=0$ for the K\"ah\-ler form
$\,\omega\,$ (Corollary~\ref{sympl}).

Finally, as a partial answer to the
above question, we show that $\,v\,$ must be 
real hol\-o\-mor\-phic when $\,M\,$ admits no nontrivial
hol\-o\-mor\-phic quadratic differentials (Theorem~\ref{holtf}) and hence,
in particular, when $\,M\,$ is almost homogeneous and has
nonzero Euler characteristic 
(Corollary~\ref{lochg}). We also provide, in Theorem~\ref{holof}, 
a pseu\-\hbox{do\hs-}Riem\-ann\-i\-an analog of a Hodge decomposition for 
the $\,1$-form $\,g(v,\,\cdot\,)$, and observe that it coincides
with the Riemannian Hodge decomposition of $\,g(v,\,\cdot\,)$ relative to
any Riemannian K\"ah\-ler metric $\,h\,$ on $\,M\nh$, if such $\,h\,$ exists.
  
\renewcommand{\thetheorem}{\thesection.\arabic{theorem}}
\section{Preliminaries}\label{pr}
\setcounter{equation}{0}
Manifolds and mappings are assumed smooth, the former also connected. 

Let $\,(M\nh,g)\,$ be a pseu\-do\hs-Riem\-ann\-i\-an manifold. We write
$\,\beta\,\sim\,B\,$ when tensor fields $\,\beta\,$ and $\,B\,$ of types 
$\,(0,2)\,$ and $\,(1,1)\,$ are related by $\,\beta=g(B\hs\cdot\,,\,\cdot\,)$.
Thus, $\,g\,\sim\,\mathrm{Id}$. On a pseu\-do\hs-\hn K\"ahler manifold
\begin{equation}\label{osj}
\omega\,\sim\hs J\,\mathrm{\ for\ the\ K}\ddot{\mathrm{a}}\mathrm{h\-ler\ form\
}\,\,\omega\,\mathrm{\ and\ the\ com\-plex}\hyp\mathrm{struc\-ture\ tensor\
}\hs J\hh.
\end{equation}
If $\,\beta\,\sim\hs B$, as defined above, and we set
\begin{equation}\label{aen}
A\,=\,\nabla\nh v
\end{equation}
for a fixed vector field $\,v$, one easily sees that
\begin{equation}\label{lvb}
\pounds\hskip-1pt_v\w\hs\beta\,
\sim\,\nabla\hskip-2.3pt_v\w B\,+\,BA\,+\,A\nh^*\nnh B\hh,
\end{equation}
$A\nh^*$ being the pointwise $\,g$-ad\-joint of $\,A$. Two obvious special
cases are
\begin{equation}\label{lvg}
\mathrm{a)}\hskip7pt\pounds\hskip-1pt_v\w\hs g\,
\sim\,A\,+A\nh^*,\qquad\mathrm{b)}\hskip7pt\pounds\hskip-1.5pt_v\w\hs\omega\,
\sim\,J\nnh A\,+\nh A\nh^*\hskip-2.2ptJ,
\end{equation}
the latter in the pseu\-do\hs-\hn K\"ahler case. Assuming 
(\ref{aen}), we obviously 
get
\begin{equation}\label{dxi}
d\hh[g(v,\,\cdot\,)]\,\sim\,A\,-A\nh^*\nnh.
\end{equation}
If $\,\beta\,\sim\,B\,$ for a real differential $\,2$-form on a complex
manifold,
\begin{equation}\label{one}
\begin{array}{l}
\beta\,\mathrm{\ is\ a\ }\,(1,1)\hyp\mathrm{form\ if\ and\ only\ if\
}\,[J,\hn B]\hn=0\hh\mathrm{,\ as\ both}\\
\mathrm{conditions\ are\ clearly\ equivalent\ to\ 
}\,\,\beta(J\hs\cdot\,,J\hs\cdot\hs)\,=\beta\hh.
\end{array}
\end{equation}
Given a pseu\-do\hs-Riem\-ann\-i\-an manifold $\,(M\nh,g)$, using $\,J$,
\begin{equation}\label{tre}
\mathrm{we\ treat\ }\,T\nh M\,\mathrm{\ as\ a\ complex\ vector\ bundle.}
\end{equation}
Cartan's homotopy formula 
$\,\pounds\hskip-1pt_v\w=\imath_v\w\hn d+d\hs\imath_v\w$ for
$\,\pounds\hskip-1pt_v\w$ acting on differential forms
\cite[Thm.\, 14.35, p.\,372]{lee} and the Leib\-niz rule
$\,\pounds\hskip-1pt_v\w[\nabla\nh\varTheta]
=[\pounds\hskip-1pt_v\w\nabla]\hs\varTheta
+\nabla\nh[\pounds\hskip-1pt_v\w\varTheta]\,$ 
imply that, for any vector field $\,v\,$ on a manifold,
\begin{equation}\label{lvo}
\pounds\hskip-1.5pt_v\w\hs\omega\,=\,d\hh[\hh\omega(v,\,\cdot\,)]\,\mathrm{\ if\
}\,\omega\,\mathrm{\ is\ a\ closed\ differential\ form,}
\end{equation}
while, whenever $\,v\,$ happens to be af\-fine relative to a connection 
$\,\nabla\nh$,
\begin{equation}\label{lvt}
\nabla\nh[\pounds\hskip-1pt_v\w\varTheta]=0\,\mathrm{\ \ if\ 
}\,\varTheta\,\mathrm{\ is\ a\ tensor\ field\ with\ }\,\nabla\nh\varTheta=0\hh.
\end{equation}
\begin{remark}\label{cstrk}Any con\-stant-rank twice-co\-var\-i\-ant symmetric
tensor field $\,\beta\,$
on a manifold 
has the same algebraic type at
all points: its positive and negative indices, being lower
sem\-i\-con\-tin\-u\-ous, with a constant sum, must be locally constant.
\end{remark}
\begin{remark}\label{holom}Due to the Leib\-niz rule, for any vector field
$\,v\,$ on a pseu\-do\-K\"ah\-ler manifold, 
$\,\pounds\hskip-1pt_v\w J=[J,A]$, where $\,A=\nabla\nh v$, so that $\,v\,$ is
real hol\-o\-mor\-phic if and only if $\,\nabla\nh v\,$ commutes
with $\,J$. On the other hand, hol\-o\-mor\-phic com\-plex-val\-ued functions
$\,\phi\,$ on a complex manifold $\,M\,$ are characterized by the
Cau\-chy-Rie\-mann condition $\,(d\phi)J=i\hs d\phi$, where
$\,(d\phi)J\,$ denotes the composite bundle morphism 
$\,T\nh M\to T\nh M\to M\times\bbC$.
\end{remark}

\section{The $\,\ddb\,$ property}\label{dd} 
\setcounter{equation}{0}
Every compact
complex manifold admitting a Riemannian K\"ahler metric has the
following $\,\ddb$ {\it property}, also referred to as {\it the\/ $\,\ddb\hs$
lemma\/} \cite[Prop.\,6.17 on p.\,144]{voisin}: given integers
$\,p,q\ge0$, any closed $\,\pl\hs$-exact or
$\,\cro\hs$-exact $\,(p,q)$-form equals $\,\ddb\hh\lambda\,$ for some
$\,(p-1,q-1)$-form $\,\lambda$. Then, since the exactness of a $\,(p,0)$-form
amounts to its $\,\pl\hs$-exact\-ness, and implies its closedness,
\begin{equation}\label{exz}
M\,\mathrm{\ admits\ no\ nonzero\ exact\ }\,(p,0)\hyp\,\mathrm{\ or\ 
}\,(0,p)\hyp\mathrm{forms.}
\end{equation}
As a special case, on a compact complex $\,\ddb$ manifold $\,M\nh$,
\begin{equation}\label{spc}
\mathrm{every\ exact\ real\ }\,\,(1,1)\hyp\mathrm{form\ 
}\,\alpha\,\mathrm{\ equals\ }\,i\hh\ddb\hh\phi\,\mathrm{\ for\ some\
}\,\phi:M\to\bbR\hh.
\end{equation}
In fact, writing $\,\alpha=d\hs\xi=\pl\hs\xi+\,\cro\hs\xi$, with a real
$\,1$-form $\,\xi$, we see that 
$\,\pl\hs\xi$ and $\,\,\cro\hs\xi$ are both closed: 
$\,d\hs\pl\hs\xi=\,\cro\hs\pl\hs\xi=\,\cro\hs d\hs\xi=\,\cro\alpha=0
=\pl\alpha=\,\ddb\hs\xi=d\,\hs\cro\hs\xi$, since $\,d=\pl\hs+\hs\cro$, while
$\,\pl^2\nnh=\,\cro{}^2\nh=0\,$ and $\,\alpha$, being closed, has
$\,\pl\alpha=\,\cro\alpha=0$. Also, as $\,\alpha\,$ is a $\,(1,1)$-form,
so must be $\,\pl\hs\xi$ and $\,\,\cro\hs\xi$. In view of the $\,\ddb$
property, the $\,(1,1)$-forms $\,\pl\hs\xi$ and $\,\,\cro\hs\xi$, 
being closed and $\,\pl\hs$-exact or $\,\cro\hs$-exact, equal
$\,\ddb\hh\lambda\,$ and $\,\ddb\hh\my\,$ for some functions
$\,\lambda,\my$, so that $\,\alpha=\mathrm{Re}\,\alpha=i\hh\ddb\hh\phi$,
where $\,\phi=\mathrm{Im}\,(\lambda+\my)$. 

On the other hand, whenever $\,\phi:M\to\bbR$,
\begin{equation}\label{twi}
2i\hh\ddb\hh\phi\,=\,-d[(d\phi)J]\hh,
\end{equation}
$(d\phi)J\,$ being the composite bundle morphism
$\,T\nh M\to T\nh M\to M\times\bbR$. Consequently,
\begin{equation}\label{idd}
-\nh2\alpha(J\hs\cdot\,,\,\cdot\,)\,=\,\theta(J\hs\cdot\,,J\hs\cdot\,)\,
+\,\theta(\,\cdot\,,\,\cdot\,)\,\,\mathrm{\ if\ }\,\alpha
=i\hh\ddb\hh\phi\,\mathrm{\  and\ }\,\theta=\nabla\nh d\phi\hh.
\end{equation}
where $\,\nabla$ is any tor\-sion-free connection on $\,M\,$ with 
$\,\nabla\nnh J=0$.
Whether or not such $\,\nabla\hs$ exists, at any critical point $\,z\,$ of a
function $\,\phi:M\to\bbR$, by (\ref{twi}),
\begin{equation}\label{crt}
-\nh2\alpha(J\hs\cdot\,,\,\cdot\,)\,=\,\theta(J\hs\cdot\,,J\hs\cdot\,)\,
+\,\theta(\,\cdot\,,\,\cdot\,)\,\,\mathrm{\ if\ }\,\alpha
=i\hh\ddb\hh\phi\,\mathrm{\  and\ }\,\theta=\mathrm{Hess}\hn_z\w \phi\hh.
\end{equation}
\begin{lemma}\label{expzr}A compact complex manifold\/ $\,M\,$ satisfying the
special case\/ {\rm(\ref{spc})} of the\/ $\,\ddb\hs$ condition admits no 
nonzero con\-stant-rank real-val\-ued exact\/ $\,(1,1)$-forms.
\end{lemma}
\begin{proof}Applying (\ref{spc}) and (\ref{crt}) to a con\-stant-rank 
real-val\-ued exact $\,(1,1)$-form $\,\alpha$, at a point $\,z\in M\,$ where 
$\,\phi\,$ assumes its maximum (or, minimum) value, we see that the symmetric
$\,2$-ten\-sor field $\,\alpha(J\hs\cdot\,,\,\cdot\,)\,$ is positive (or,
negative) sem\-i\-def\-i\-nite at $\,z$. Remark~\ref{cstrk} applied to
$\,\beta=\alpha(J\hs\cdot\,,\,\cdot\,)$, which has constant rank (equal to
the rank of $\,\alpha$), implies both positive and 
negative sem\-i\-def\-i\-nite\-ness of $\,\beta\,$ at all points, so that
$\,\beta=0\,$ everywhere.
\end{proof}

\section{Consequences of the Hodge decomposition}\label{ch} 
\setcounter{equation}{0}
For each cohomology space $\,H^p\nnh(M,\bbC)\,$ of a compact complex manifold
$\,M\,$ with the $\,\ddb$ property, denoting by $\,H^{r,s}\nnh M\,$ the space 
of cohomology classes of closed com\-plex-val\-ued $\,(r,s)$-forms, one has
the Hodge decomposition 
\begin{equation}\label{hdc}
H^p\nnh(M,\bbC)=H^{p,0}\nnh M\oplus H^{p-1,1}\nnh M\oplus\ldots
\oplus H^{1,p-1}\nnh M\oplus H^{0,p}\nnh M\hh.
\end{equation}
See, e.g.,
\cite[p.\,296, subsect.\,(5.21)]{deligne-griffiths-morgan-sullivan}. 
\begin{lemma}\label{expar}Let\/ $\,\zeta\hh$ be an exact\/ 
$\,\nabla\nh$-par\-al\-lel com\-plex-val\-ued\/ $\,p$-\hn form
on a compact complex\/ $\,\ddb$ manifold\/ $\,M\,$ with a 
tor\-sion-free connection\/ $\,\nabla$ such that\/ $\,\nabla\nnh J=0$.
\begin{enumerate}
\item[(i)]$\zeta\hh$ has zero\/ $\,(p,0)\,$ and\/ $\,(0,p)$ components.
\item[(ii)]The\/ $\,(r,s)\,$ components of\/ $\,\zeta$, $\,r+s=p$, are all
exact and\/ $\,\nabla\nh$-par\-al\-lel.
\item[(iii)]$\zeta=0\,$ when\/ $\,p=2$.
\end{enumerate}
\end{lemma}
\begin{proof}
The decomposition of the bundle of com\-plex-val\-ued exterior $\,p\hs$-forms
on $\,M\,$ into its $\,(r,s)\,$ summands, with $\,r+s=p$, is invariant under
parallel transports, since $\,J\,$ uniquely determines the decomposition and
$\,\nabla\nnh J=0$. The components $\,\zeta\hh^{r,s}$ of the decomposition
of $\,\zeta\,$ are thus all $\,\nabla\nh$-par\-al\-lel, and hence closed.
The resulting cohomology relation
$\,\sum_{\,r,s}\w[\zeta\hh^{r,s}]=[\zeta]=0\,$ gives, by (\ref{hdc}),
$\,[\zeta\hh^{r,s}]=0$ whenever $\,r+s=p$, proving (ii), while (i) follows
from (ii) and (\ref{exz}). Lemma~\ref{expzr} and (i) -- (ii) now easily yield 
(iii). 
\end{proof}
We have an obvious consequence of (\ref{lvo}) -- (\ref{lvt}) and 
Lemma~\ref{expar}(iii).
\begin{corollary}\label{sympl}Let\/ $\,v\,$ be an af\-fine vector field on a
compact pseu\-do\hs-\hn K\"ahler\/ $\,\ddb$  manifold\/ $\,(M\nh,g)\,$ with 
the K\"ah\-ler form\/ $\,\omega=g(J\hs\cdot\,,\,\cdot\,)$. Then\/ 
$\,\pounds\hskip-1.5pt_v\w\hs\omega=0$.
\end{corollary}

\section{The case of Kil\-ling fields}\label{ck} 
\setcounter{equation}{0}
The paper \cite{derdzinski-terek} provides two different proofs of the
fact that, {\it on a compact pseu\-do-K\"ah\-ler\/ $\,\ddb$ manifold, all 
Kil\-ling fields are real hol\-o\-mor\-phic.}

Our preceding discussion gives rise to two more simple proofs of this fact.
As
\begin{equation}\label{lgj}
\pounds\hskip-1pt_v\w\hn[\hs g(\hn J\hs\cdot\,,\,\cdot\,)]\,
=\,g(\pounds\hskip-1.5pt_v\w J\hs\cdot\,,\,\cdot\,)\,\mathrm{\ when\
}\,\pounds\hskip-1.5pt_v\w g=0\hh,
\end{equation}
and $\,\omega=g(J\hs\cdot\,,\,\cdot\,)$, one new
proof comes directly from Corollary~\ref{sympl}.

For the other new proof, note that $\,\zeta=\pounds\hskip-1.5pt_v\w\hs\omega$,
being
exact and parallel by (\ref{lvo}) and (\ref{lvt}), must -- due to
Lemma~\ref{expar}(i) -- be a $\,(1,1)$-form. Since (\ref{lgj}) gives 
$\,\pounds\hskip-1.5pt_v\w\hs\omega\,\sim\hs\pounds\hskip-1.5pt_v\w J$,
(\ref{one}) implies that $\,J\,$ and $\,C=\pounds\hskip-1.5pt_v\w J\,$
commute. However, they also anti\-com\-mute: 
$\,0=-\pounds\hskip-1pt_v\w\hs\mathrm{Id}
=\pounds\hskip-1.5pt_v\w J\hs^2\nh=CJ+J\nh C$. Consequently,
$\,\pounds\hskip-1.5pt_v\w J=C=0$.

The first proof in \cite{derdzinski-terek}, at the end of Sect.\,3, procceds
as follows. By (\ref{lvt}) above,
$\,\pounds\hskip-1.5pt_v\w J\,$ is parallel, and 
hence so is the $\,(2,0)$-form 
$\,\zeta=g(\hn\pounds\hskip-1.5pt_v\w J\hs\cdot\,,\,\cdot\,)
-ig((\hn\pounds\hskip-1.5pt_v\w J)J\hs\cdot\,,\,\cdot\,)$, which makes
$\,\zeta\,$ closed as well. However, $\,\zeta\,$ is also 
$\,\pl\hs$-exact, namely -- see \cite[Lemma 3.4]{derdzinski-terek} -- 
equal to $\,\pl\hs[\hs g(Jv,\,\cdot\,)-ig(v,\,\cdot\,)]$. Combined with its
closedness and the $\,\ddb$ property, this, 
according to \cite[Lemma 3.1]{derdzinski-terek}, gives $\,\zeta=0$.

The second proof in \cite{derdzinski-terek}, in Sect.\,4, uses the
parallel $\,(2,0)$-form $\,\zeta\,$ of the last paragraph. By (\ref{lvo}), 
$\,\mathrm{Re}\,\zeta=\pounds\hskip-1.5pt_v\w\hs\omega\,$ is exact.
This makes $\,[i\zeta]\in H^{2,0}\nnh M\,$ a real cohomology class, equal
to its conjugate lying in $\,H^{0,2}\nnh M\nh$, and so $\,\zeta=0$, as
$\,H^{2,0}\nnh M\nh$, $\,H^{0,2}\nnh M\,$ and $\,H^{1,1}\nnh M\,$ are
direct summands of $\,H^2\nnh(M,\bbC)$,

\section{The four components of the covariant derivative}\label{mo} 
\setcounter{equation}{0}
On a pseu\-do\hs-\hn K\"ahler manifold $\,(M\nh,g)$, the operation
$\,A\mapsto J\nh AJ\,$ applied to bundle morphisms $\,A:T\nh M\to T\nh M\,$
obviously commutes with $\,A\mapsto A\nh^*$ and, as both are involutions,
every $\,A\,$ is decomposed into four components (com\-plex-lin\-e\-ar
self-ad\-joint, com\-plex-lin\-e\-ar skew-ad\-joint, anti\-lin\-e\-ar 
self-ad\-joint, anti\-lin\-e\-ar skew-ad\-joint). In the case where
$\,A=\nabla\nh v\,$ for an af\-fine vector field $\,v\,$ on a 
compact pseu\-do\-K\"ahler $\,\ddb$ manifold $\,(M\nh,g)$, two of the 
four components -- the first and last ones -- are absent, according to the
assertions (\ref{csq}-b) and (\ref{csq}-c) below, while the third one has 
rather special algebraic and differential properties, cf. (\ref{csq}-d).
\begin{theorem}\label{fourc}
Let\/ $\,v\,$ be an af\-fine vector field on a 
compact pseu\-do\hs-\hn K\"ahler\/ $\,\ddb$ manifold\/ $\,(M\nh,g)$. Then,
for\/ $\,A=\nabla\nh v$,
\begin{equation}\label{csq}
\begin{array}{rl}
\mathrm{a)}&A\nh^*\hs=\,J\nh AJ\hh,\\
\mathrm{b)}&A\,-\hs A\nh^*\,\mathrm{\ commutes\ with\ }\,J\hh,\\
\mathrm{c)}&A\,+\hs A\nh^*\,\mathrm{\ anti\-com\-mutes\ with\ }\,J\hh,\\
\mathrm{d)}&A\,+\hs A\nh^*\,\mathrm{\ is\ parallel,\ 
and\ nil\-po\-tent\ at\ every\ point,}\\
\mathrm{e)}&\mathrm{div}\hskip1.7ptv\,=\,0\hh,\\
\mathrm{f)}&v\,\mathrm{\ is\ a\ Kil\-ling\ field\ if\ and\ only\ if\ it\ is\
real}\hyp\mathrm{hol\-o\-mor\-phic.}
\end{array}
\end{equation}
\end{theorem}
\begin{proof}Corollary~\ref{sympl} and (\ref{lvg}.b) yield (\ref{csq}-a), while
(\ref{csq}-a) trivially implies \hbox{(\ref{csq}-b)} -- (\ref{csq}-c). Next,
(\ref{lvt}) and (\ref{lvg}.b) prove the first part of (\ref{csq}-d). Thus, 
$\,2\,\mathrm{div}\hskip1.7ptv=2\,\mathrm{tr}\hskip1.7ptA
=\mathrm{tr}\hskip2pt(\nh A\hs+\hn A\nh^*)\,$ is constant, and has zero integral,
which gives (\ref{csq}.e) and the equality 
$\,\mathrm{tr}\hskip2pt(\nh A\hs+\hn A\nh^*)^k\nh=0\,$ for $\,k=1$. The same
equality for $\,k\ge2\,$ now follows: setting
$\,C=(A\hs+\hn A\nh^*)^{k-1}\nh$, we get
$\,\mathrm{tr}\hskip2pt(\nh A\hs+\hn A\nh^*)^k\nh=
\mathrm{tr}\hskip2.7pt(C\nnh A\hs+\hn C\nnh A\nh^*)
=2\,\mathrm{tr}\hskip2.7ptC\nnh A$.
(Note that $\,\mathrm{tr}\hskip2.7ptC\nnh A\nh^*\nh
=\mathrm{tr}\hskip2.7pt(C\nnh A\nh^*)^*\nh
=\mathrm{tr}\hskip1.7ptAC=\mathrm{tr}\hskip2.7ptC\nnh A$.)
Due to the first part of \hbox{(\ref{csq}-d)}, $\,C\,$ is parallel and
$\,\mathrm{tr}\hskip2.7ptC\nnh A\,$ constant. As the constant 
$\,\mathrm{tr}\hskip2.7ptC\nnh A\,$ equals
$\,C^{\hs i}_{\nh k}v^{\hh k}{}\nnh_{,i}\w
=(C^{\hs i}_{\nh k}v^{\hh k})\nnh_{,i}\w=\hs\mathrm{div}\hskip1.7ptCv$, it 
has zero integral. The zero-in\-te\-gral constant 
$\,\mathrm{tr}\hskip2pt(\nh A\hs+\hn A\nh^*)^k$ thus equals $\,0$, which 
which yields the induction step, proving the remainder of (\ref{csq}-d).
Finally, (\ref{csq}-f) follows from
Corollary~\ref{sympl}: since $\,\pounds\hskip-1.5pt_v\w\hs\omega=0\,$ and 
$\,\omega=g(J\hs\cdot\,,\,\cdot\,)$, the condition
$\,\pounds\hskip-1.5pt_v\w g=0\,$ is equivalent to
$\,\pounds\hskip-1.5pt_v\w J=0$.
\end{proof}

\section{Two hol\-o\-mor\-phic covariant tensors}\label{th} 
\setcounter{equation}{0}
Equation (\ref{gvd}) in Theorem~\ref{holof} constitutes both a
pseu\-\hbox{do\hs-}Riem\-ann\-i\-an analog of a
har\-mon\-ic-co\-ex\-act Hodge decomposition for 
the $\,1$-form $\,g(v,\,\cdot\,)$, and the actual Riemannian Hodge 
decomposition of $\,g(v,\,\cdot\,)\,$ relative to any Riemannian K\"ah\-ler 
metric $\,h$, as long as one exists on $\,M\nh$. This follows since 
$\,(d\phi)J\,$ and $\,\xi\hh$ appearing in (\ref{gvd}) are, respectively,
$\,h$-co\-ex\-act and $\,h$-har\-mon\-ic, with the exact part absent, for 
every pseu\-do\hs-Riem\-ann\-i\-an K\"ah\-ler metric $\,h\,$ which either 
equals our $\,g$, or is positive definite (and, in the case of $\,g$, by
$\,g${\it-har\-mon\-ic\/} one means closed and $\,g$-coclos\-ed). In fact,
$\,(d\phi)J\,$ is the 
$\,h$-di\-ver\-gence of $\,-\phi\hh h(J\hs\cdot\,,\,\cdot\,)$, for the
$\,h$-K\"ah\-ler form $\,h(J\hs\cdot\,,\,\cdot\,)$. Furthermore, we prove
below that both $\,\xi\hs$ and $\,\xi\hn J\,$ are closed, 
while, given any real $\,1$-form $\,\xi\hs$ on a complex manifold,
\begin{equation}\label{ofh}
\xi-i\hh\xi\hn J\,\mathrm{\ is\ hol\-o\-mor\-phic\ whenever\ 
}\,d\hs\xi=d(\xi J)=0\hh,
\end{equation}
since, in general, for a $\,(p,0)$-form $\,\zeta$, closedness ($d\hs\zeta\nh
=0$) obviously yields hol\-o\-mor\-phic\-i\-ty ($\hs\hs\cro\hs\zeta\nh=0$).
Being hol\-o\-mor\-phic, $\,\xi-i\hh\xi\hn J\,$ must be $\,h$-har\-mon\-ic
\cite[Ch.\,5]{ballmann} when $\,h\,$ is Riem\-ann\-i\-an. Finally, 
(\ref{gvd}) and (\ref{csq}-e) imply $\,g$-coclos\-ed\-ness of $\,\xi$.
\begin{theorem}\label{holof}Given an af\-fine vector field\/ $\,v\,$  on a 
compact pseu\-do\hs-\hn K\"ahler\/ $\,\ddb$  manifold\/ $\,(M\nh,g)$, one has
\begin{equation}\label{gvd}
g(v,\,\cdot\,)\,=\,(d\phi)J\hs+\,\xi
\end{equation}
for some\/ $\,\phi:M\to\bbR\,$ and the real part\/ $\,\xi$ of a 
hol\-o\-mor\-phic\ $\,1$-form\/ $\,\xi-i\hh\xi J$.
\end{theorem}
\begin{proof}The assertions (\ref{dxi}), (\ref{csq}.b), and (\ref{one})
applied to $\,B=A\,-A\nh^*\nh$, show that $\,d\hh[g(v,\,\cdot\,)]\,$ is an
exact $\,(1,1)$-form. Choosing $\,\phi\,$ as in (\ref{spc}) for
$\,\alpha=-d\hh[g(v,\,\cdot\,)]/2$, we now see that, by (\ref{twi}), the
$\,1$-form $\,\xi=g(v,\,\cdot\,)-(d\phi)J\,$ is closed. However,
$\,\xi J=d\phi-g(Jv,\,\cdot\,)=d\phi-\omega(v,\,\cdot\,)\,$ is also closed,
due to (\ref{lvo}) and Corollary~\ref{sympl}. 
The hol\-o\-mor\-phic\-i\-ty of $\,\xi-i\hh\xi\hn J\,$ now follows from
(\ref{ofh}).
\end{proof}
Theorem~\ref{holof} does not seems to be relevant to the question stated
in the Introduction: for instance, vanishing of the hol\-o\-mor\-phic
$\,1$-form $\,\xi-i\hh\xi J\,$ (which follows if $\,M\,$ is simply connected)
does not lead to any immediate answer. This stands in marked contrast with
the next result.
\begin{theorem}\label{holtf}Suppose that\/ $\,v\,$ is an af\-fine vector field
on a compact pseu\-do\hs-\hn K\"ahler\/ $\,\ddb$ manifold\/ $\,(M\nh,g)$. 
Then the\/ $\,(0,2)\,$ tensor field
$\,\pounds\hskip-1.5pt_v\w g\,$ is the real part of a 
hol\-o\-mor\-phic section\/ $\,\theta\,$ of the second complex symmetric power
of the complex dual of\/ $\,T\nh M\nh$, with the convention\/
{\rm(\ref{tre})}.

Consequently, $\,\pounds\hskip-1.5pt_v\w g=0\,$ if no such nonzero
hol\-o\-mor\-phic section\/ $\,\theta\,$ exists.
\end{theorem}
\begin{proof}By (\ref{lvg}.a) and (\ref{csq}-c), 
$\,[\pounds\hskip-1pt_v\w\hs g](J\hs\cdot\,,J\hs\cdot\hs)
=-\pounds\hskip-1pt_v\w\hs g$, so that 
$\,\theta=\pounds\hskip-1pt_v\w\hs g
-i[\pounds\hskip-1pt_v\w\hs g](J\hs\cdot\,,\,\cdot\,)\,$ is
com\-plex-bi\-lin\-e\-ar at every point. As it is parallel -- by (\ref{lvt})
-- its hol\-o\-mor\-phic\-i\-ty follows: Remark~\ref{holom} easily implies
that, for any (local) real hol\-o\-mor\-phic vector fields $\,w,w'\nh$, the
function $\,\theta(w,w')\,$ is hol\-o\-mor\-phic.
\end{proof}
A complex manifold $\,M\,$ is called {\it almost homogeneous\/}
\cite{oeljeklaus} if, for 
some $\,x\in M\nh$, every vector in $\,T\hskip-2.9pt_x\w\hn M\,$ is the value
at $\,x\,$ of some real hol\-o\-mor\-phic vector field.
\begin{corollary}\label{lochg}If\/ $\,g\,$ is a pseu\-do\hs-\nh K\"ahler
metric on a 
compact almost homogeneous complex\/ $\,\ddb$ manifold\/ $\,M\,$ 
with nonzero Euler characteristic\/ $\,\chi(M)$, 
then every $\,g$-af\-fine vector field\/
$\,v\,$ on\/ $\,M\,$ is real hol\-o\-mor\-phic.
\end{corollary}
In fact, let $\,\theta\,$ be the hol\-o\-mor\-phic quadratic differential 
mentioned in
Theorem~\ref{holtf}. Thus, $\,\theta(w,w')\,$ is constant for any 
real hol\-o\-mor\-phic vector fields $\,w,w'\nh$. If $\,\theta\,$ were
nonzero -- everywhere, due to its being parallel by (\ref{lvt}) -- choosing
$\,x\,$ as above and a real hol\-o\-mor\-phic vector field $\,w\,$ with
$\,\theta(w,w)\ne0\,$ at $\,x\,$ we would get $\,w\ne0\,$ everywhere, and
hence, by the Poin\-ca\-r\'e-Hopf theorem, $\,\chi(M)=0$.

\section*{Conflict of interest statement}
\setcounter{equation}{0}
The author states that there is no conflict of interest.

\end{document}